\documentclass[12pt]{article}
\usepackage{amssymb,latexsym,amsmath,epsfig,pdfsync, amsthm} 

\usepackage{amsmath,amssymb,enumerate,bbm}
\newtheorem{theorem}{Theorem}[section]

\newtheorem{lemma}[theorem]{Lemma}

\newtheorem{definition}[theorem]{Definition}
\newtheorem{remark}[theorem]{Remark}

\newtheorem{hypothese}[theorem]{Hypothese}
\numberwithin{equation}{section}

\begin{document}

\title{The Hautus test for non-autonomous linear evolution equation}
\author{Duc-Trung Hoang\\
Institute Mathematics of Bordeaux, France. \\
dthoang@math.u-bordeaux.fr}
\date{}
\maketitle
\abstract{
  In this paper, we investigate the Hautus test for evolution equation with the operators depending on time.}

\section{Introduction}
Controllability and observability are basis concepts in system theory and control theory. They are important structural properties which have close relationships with the stability of state feedback controllers abd state observers. In this paper, we will study the controllability, the observabilty, the duality between these two concepts for the non autonomous linear system. These properties were studied well for the autonomous system.
Let $H$ be a Hilbert space. Considering $U(t,s)$ the evolution family of two variables generating by the family of operators $A(t)$: $A(t): D(A(t)) \mapsto H$. Let $U$ be another Hilbert space and suppose $C: H \rightarrow U$ is a linear operator.  We consider the system :
\begin{equation}
  \label{1}
\left\{
  \begin{array}{lcl}
  x'(t) + A(t) x(t) &=& B(t)u(t) \\ x(0) &=& x_0 \\ y(t) &=& Cx(t)
  \end{array}\right.
\end{equation}
For simplicity, we denote the above system as $(A(t),B(t),C)$.  We always assume that the family of operator $A(t)$ is bounded from $H \mapsto H$. The solution is defined as
\[x(t) = U(t,0)x_0 + \int_{0}^{t}U(t,s)B(s)u(s) ds , \quad t \geq s\]

\begin{definition}

\end{definition}
A family $\{U(t,s)_{t,s}\}$  operators is called an evolution family if it satisfies the following conditions : \\
(i) $U(t,t)x = x$ for all $t \geq 0$ and $x \in H$; \\
(ii) $U(t,s) = U(t,r).U(r,s)$ for all $ t \geq r \geq s \geq 0$
Such an evolution family is called continous if there exist $M, \omega > 0$ such that \\
(iii) $\|U(t,s)\| \leq Me^{\omega(t-s)}$ \\
(iv) $U(t,s)x$ is jointly continuous with respect to $t,s$ and $x$

\begin{definition}
The system (\ref{1}) is said to be exactly controllable at time $\tau$ if for every $x_0, x_1$ in $H$, there exist $u \in L^2(0,\tau;U)$  such that the solution satisfy $x_1 = x(\tau)$
\end{definition}

\begin{definition}
The system (\ref{1}) is said to be exactly null controllable at time $\tau$ if for every $x_0$ in $H$, there exist $u \in L^2(0,\tau;U)$  such that the solution satisfy $x(\tau) = 0$
\end{definition}

\begin{definition}
The system (\ref{1}) is said to be observable in $[0,\tau]$ if the map $ O: H \rightarrow L^2(0,\tau;U) : x_0 \rightarrow y(.) $ is injective. 
\end{definition}

The definition express the fact that we can recover uniquely the initial state from a knowledge of the output $y(.)$ in the time interval $[0, \tau]$. When the system is observable, we refer to $(C,A(t))$ as an observable pair. For one variable $s$ fixed, $A(s)$ generate a strongly continuous semigroup $T_s(t)$.
We assume that the domaine $D(A(t))$ is densed in $H$ and independent of $t$. We consider the adjoint system

\begin{equation}
  \label{2}
\left\{
  \begin{array}{lcl}
  z'(t) - A(t)^* z(t) &=& 0 \\ z(\tau) &=& z_{\tau} \\ h(t) &=& B(t)^*x(t)
  \end{array}\right.
\end{equation}

Russell and Weiss (\cite{russell}) showed that a necessary condition for exact observability of exponentially systems is the following Hautus test : There exits a constant $m > 0$ such that for every $s \in \mathbb{C}_{-}$ and every $x \in D(A)$. The Hautus test can be use for approximate observability of exponentially stable systems [3], for polynomially stable system [4], for exact observability of strongly stable Riesz-spectral systems with finite dimensional output spaces [5], and for exponentially stable $C_0-$groups [6]

\[\|(sI-A)x\|^2+|Re s|\|Cx\|^2 \geq |Re s|^2\|x\|^2\]
where $\mathbb{C}_{-}$ denotes the open left half plane.

\section{Duality of controllability and observability}

Let $A(t)$ be such that the uncontrolled initial value problem 
\begin{equation}
  \label{eq:evol-eq-without-control}
\left\{
  \begin{array}{lcl}
  x'(t) + A(t) x(t) &=& 0 \\ x(0) &=& x_0
  \end{array}\right.
\end{equation}
admits a evolution (solution) family $U(t, s)$.  We observe that
\[
U(t, s+h) - U(t, s) = U(t, s+h) [ I - U(s+h, s) ]
\]
and so, dividing by $h>0$ and letting $h \to 0+$,
\[
\tfrac{d}{ds{+}} U(t, s) = - U(t, s) A(s)
\]
Under mild extra conditions, this derivative will exist in both directions. 
Now we take adjoints:
\[
\tfrac{d}{ds{+}} \langle U(\tau, s)^* z_\tau ,x \rangle 
= 
\tfrac{d}{ds{+}}  \langle  z_\tau , U(\tau, s) x \rangle 
=
\langle z_\tau , - U(\tau, s) A(s) x \rangle
=
\langle - A(s)^* U(\tau, s)^* z_{\tau}, x \rangle.
\]
This holds for all $x$, so we may drop duality pairing and obtain that
 $z(t) := U(\tau, t)^* z_\tau$ will solve the dual final time problem
\begin{equation}
  \label{eq:dual-evol-eq}
\left\{
  \begin{array}{lcl}
  z'(t) {\mathbf-} A(t)^* z(t) &=& 0 \\ z(\tau) &=& z_\tau
  \end{array}\right.
\end{equation}

\subsection{Duality}
Now consider
\begin{equation}
  \label{eq:evol-eq-with-control}
\left\{
  \begin{array}{lcl}
  x'(t) + A(t) x(t) &=& B(t)u(t) \\ x(0) &=& 0
  \end{array}\right.
\end{equation}
and assume, that the map $\Psi_\tau: L_2(0, \tau; U) \to H$ defined by
$u \mapsto x(\tau)$ is continuous (i.e. that $B(t)_{0\le t \le \tau}$
is an admissible family of control operators for the evolution
equation). Assume further {\bf exact controllability}, i.e.  that for
{any} $x_\tau \in X$, we can find some $u\in L_2(0, \tau; U)$ such
that the solution of the initial value problem
(\ref{eq:evol-eq-with-control}) satisfies $x(\tau) = x_\tau$. Then
$\Psi_\tau$ is bounded and surjective. \\

$\Psi_\tau u = \int_0^\tau U(\tau, s)
B(s) u(s)\,ds$ and so $(\Psi_\tau^* x^*)(t) = B(t)^* U(\tau, t)^* x^*$.

According to the ``standard lemma'', $\Psi_\tau$ is surjective iff
$\Psi_\tau^*$ allows lower estimates, i.e. iff
\[
  \delta \|x^*\| \le \|B(t)^* U(\tau, t)^* x^*\|_{L_2(0, \tau)}
\]

slow solution ``a la main'' 

By the open mapping theorem we
then have a constant $C>1$ such that $ \|u\|_{L_2} \le C
\|x_\tau\|$.  We can therefore simply let $z_\tau = x_\tau$ in
(\ref{eq:dual-evol-eq}), and consider 
\[
\tfrac{d}{dt} \langle x(t) , z(t) \rangle 
= 
\langle B(\cdot) u(\cdot) - A(\cdot) x(\cdot) , z(t) \rangle + \langle x(\cdot) ,  A(\cdot)^* z(\cdot) \rangle
= \langle B(\cdot) u(\cdot), z(t)\rangle.
\]
Integrating from $0$ to $\tau$ (recall $x(0)=0$), one obtains 
\begin{equation}
  \label{eq:key}
\| x_\tau \|_H^2 = \| z_\tau \|_H^2 = \int_0^\tau \langle  u(s), B(s)^* z(s)\rangle\,ds 
\end{equation}
Now, by Cauchy-Schwarz and the hypothesis $\|{u}\|_{L_2} \le C
\|z_\tau\| $
\[
\| z_\tau \|_H^2
\le  \|u\|_{L_2} \|B(\cdot)^* z(\cdot) \|_{L_2(0, \tau; H)} 
\le C \|z_{\tau}\|_H \Bigl(\int_0^\tau \|B(s)^* z(s) \|_H^2\,ds\Bigr).
\]
Dividing by $\|z_\tau\|$, this gives the ``observability estimate''
of the adjoint problem (\ref{eq:dual-evol-eq}), that is, the estimate
\begin{equation}
  \label{eq:observ-estim}
\|z_\tau \|_H 
\le C  \Bigl(\int_0^\tau \|B(s)^* z(s) \|_H^2\,ds\Bigr).
\end{equation}
For the converse direction we assume (\ref{eq:observ-estim}),
i.e. {\bf exact observability} of the dual system. We aim to obtain
surjectivity of $\Psi_\tau$.  

\begin{theorem}
The system is exactly controllable on $0 \leq \tau < \infty$ if and only if there exists $C > 0$ such that for all $x \in H$, we have 

\begin{equation}
  \label{eq:observ-estim}
\|x_\tau \|_H 
\le C  \Bigl(\int_0^\tau \|B(s)^* U(\tau,s)x \|_H^2\,ds\Bigr)
\end{equation}

\end{theorem}

For the converse,  we define the controllability Gramian as $W_t = \Psi_\tau\Psi_\tau^* = \int_{0}^{t}U(t,s)B(s)B(s)^*U(t,s)^* ds$ be the operator depending on $t$.
Now assuming that  $\|z_\tau \|_H 
\le C  \Bigl(\int_0^\tau \|B(s)^* z(s) \|_H^2\,ds\Bigr)$. We have
\[
\|z_\tau \|_H 
\le C  \Bigl(\int_0^\tau \|B(s)^* U(\tau,s)^*z_\tau \|_H^2\,ds\Bigr)
\]
Or
\[
\|z_\tau \|^2_H 
\le C^2 \|\Psi_\tau^*z_\tau \|^2 = C^2\langle \Psi_\tau^*z_\tau, \Psi_\tau^*z_\tau \rangle  = C^2\langle \Psi_\tau\Psi_\tau^*z_\tau, z_\tau \rangle = \langle W_\tau z_\tau, z_\tau \rangle 
\]
Hence, we conclude that $W_\tau$ is self-adjoint, injective and coercive operator. Then  $W_\tau$ is boundedly invertible. Hence, Im$(W_\tau)$ $= D(W_\tau)^{-1} = H)$. 
This implies Im$(\Psi_\tau)$ $= H$ because $H = $ Im$(W_\tau)$ $\subset$ Im$(\Psi_\tau)$. This indicates the controllability of the initial system
 
\subsection{Necessary condition}
If we take $G(t,s)= U(t,s)B(s)B(s)^*U(t,s)^*$ be the function of two variables $s,t$ with $0 \leq s \leq t$. Then
\[ \frac{d}{dt}G(t,s) = A(t)G(t)+G(t)A(t)^*\]
Now if we take the integral from $0$ to $t$ with respected to the variable $s$, we have:
The controllability $W(t)$ is the unique solution of the equation

\[\frac{d}{dt}W(t)=A(t)W(t)+W(t)A(t)^*+B(t)B(t)^*\] 
Noting that the operator $W_t=\Psi_t\Psi_t^*$. We assume that $W_\tau$ is not invertible. Since, $W_\tau \geq 0$, there exists the sequence ${z_n} \in H $ such that $\|z_n\| =1 $ and $\langle z_n , W_\tau z_n \rangle \rightarrow 0$. It follows that 
\[
\int_0^tau\|U(\tau,s)^*B(s)^*z_n\|^2 dt \rightarrow 0 
\]

We also have a noting that the control function $u(t)$ and the out put function satisfy the following 
\[\int_{0}^{\tau} \langle u(t) , y(t) \rangle=0\]

\subsection{Null controllability}

The system 

\begin{equation}
  \label{eq:evol-eq-with-control}
\left\{
  \begin{array}{lcl}
  x'(t) + A(t) x(t) &=& B(t)u(t) \\ x(0) &=& 0
  \end{array}\right.
\end{equation}

is exactly null controllable on $[0,\tau]$ if for all $x_0$, we can find $ u\in L_2(0,\tau; U) $ such that 
\[ 0 = U(\tau,0)x_0 + \int_{0}^{\tau}U(\tau,s)B(s)u(s)ds\]
\[ Ran(U(\tau,s)) \subset Ran(u \rightarrow \int_{0}^{\tau}U(\tau,s)B(s)u(s)ds)\]

We define the operator $S : x_0 \rightarrow U(\tau,s)x_0$ and $T: u \rightarrow \int_{0}^{\tau}U(\tau,s)B(s)u(s)ds$.

\begin{lemma}( see \cite{tucsnack})
Suppose that $Z_1, Z_2, Z_3$ are Hilbert spaces, the operators $F \in L(Z_1, Z_3)$ and $G \in L(Z_2, Z_3)$. Then the following statements are equivalent: \\
(a) $Ran(F) \subset Ran(G)$ \\
(b) There exists a $c >0$ such that $\|F^*z\|_{Z_1} \leq c\|G^*z\|_{Z_2}$ for all $z \in Z_3$ \\
(c) There exist an operator $U \in L(Z_1,Z_2)$ such that $F = GU$
\end{lemma}

We will assume that $A(t)A(s)^{-1}$ are uniformly bounded for $s,t \in [0, \infty)$, $A(\infty)$ and $lim_{t \rightarrow \infty}\|(A(t)-A(\infty))A(0)^{-1}\|$ = 0. Then by [1] (Theorem 8.1, chapter 5.8) there exists constant $M \geq 0$ and $v >0$ such that
\[\|U(t,s)\| \leq Me^{-v(t-s)}\]

Hence the operator  $S : x_0 \rightarrow U(\tau,s)x_0$ is a bounded operator from $H \rightarrow H$

\begin{lemma}
The operator $T: u \rightarrow \int_{0}^{\tau}U(\tau,s)B(s)u(s)ds$ is bounded linear map from $L^2([0,\tau]; U) \rightarrow H$
\end{lemma}

\begin{proof}
We have:
\begin{align*}
T(u) 
= & \; \int_{0}^{\tau}U(\tau,s)B(s)u(s)ds \\
\leq & \;\int_{0}^{\tau}\|U(\tau,s)B(s)u(s)\|_H ds \\
\leq & \; \int_{0}^{\tau}Me^{-v(\tau -s)}\|B(s)\|_{L(U,H)}\|u(s)\|_H ds \\
\leq & \; Me^{-v\tau}\|B(t)\|_{L(U,H)} \int_{0}^{\tau}e^{vs}\|u(s)\|_H ds \\
\leq & \; Me^{-v\tau}\|B(t)\|_{L(U,H)}(\int_{0}^{\tau}e^{2vs}ds \int_{0}^{\tau}\|u(s)\|^2_H)^{1/2}\\
\leq & \; Me^{-v\tau}\|B(t)\|_{L(U,H)} \frac{1}{2}(\frac{e^{2v\tau} - 1}{v})^{1/2}\|u(s)\|_{L^2([0,\tau];U)}
\end{align*}
\end{proof}

Since $S,T$ are bounded operator, using the lemma there exists a constant $c > 0$ such that
\[ \|S^*x\| \leq c\|T^*x\|  \]

By computation: $T^*x = B^*(s)U(\tau, s)^*x^*$ and $S^*x = U(\tau,s)^*x^*$. Then we obtain the inequality:

\[\|B^*(s)U(\tau, s)^*x^*\| \geq \|U(\tau,s)^*x^*\| \]

%
%

\subsection{Minimum cost controls}
We have $x_\tau = U(\tau,0)x_0+\int_0^\tau U(\tau,s)B(s)u(s)ds$. It is easy to check that the control $\tilde{u} = B(s)^*U(\tau,s)^*W_\tau^{-1}(x_\tau-U(\tau,0)x_0)$ satisfies the equation.

We will indicate that $\tilde{u}$ takes the $L_2$ minimum-norm. Suppoing that both $u$ and $\tilde{u}$ satisfy the equation. Then we have
 
 \[
 \int_0^ \tau U(\tau,s)B(s)(u(s)-\tilde{u}(s)) ds = 0
  \]
  For all $\eta \in H$, we have
   
 \[
\langle \int_0^ \tau U(\tau,s)B(s)(u(s)-\tilde{u}(s)) ds , \eta \rangle = 0
  \]
  
If we choose $\eta = W_\tau^{-1}(x_\tau-U(\tau,0)x_0) $, then
      
 \[
\langle u(s)-\tilde{u}(s),  \tilde{u}(s)\rangle = 0
  \]
This impiles $\|u\|_{L^2}^2 \geq \|\tilde{u}\|_{L^2}^2$. In fact,

 \begin{align*}
 \|\tilde{u}\|_{L^2}^2 
= & \;  \int_0^ \tau \langle \tilde{u}, \tilde{u}  \rangle ds \\
= & \; \int_0^ \tau \langle B(s)^*U(\tau,s)^*W_\tau^{-1}(x_\tau-U(\tau,0)x_0) , B(s)^*U(\tau,s)^*W_\tau^{-1}(x_\tau-U(\tau,0)x_0) \rangle ds \\
= & \; \langle W_\tau (x_\tau-U(\tau,0)x_0, (x_\tau-U(\tau,0)x_0 \rangle \\
= & \; \|W_\tau^*(u)\|  
\end{align*}

\section{The Hautus test}

Observe that
\[
d/ds \left( e^{-\lambda s} U(t, s)x \right) = -\lambda e^{-\lambda s} U(t, s)x -  e^{-\lambda s} U(t, s) A(s)x
\]
and so, integrating on $[0, t]$, 
\[
- C  U(t, s) x = e^{-\lambda t} C x + \int_0^t C U(t, s) (\lambda + A(s))x  e^{-\lambda s}\,ds
\]
If we have exact observability, i.e. $\delta  \| x \|  \le  \| C U(t, 0)x  \| _{L_2(0, \tau)}$, this gives
\[
\delta  \| x \|  \le  \| Cx \|  / \sqrt{2 \Re(\lambda)} +  \| t \mapsto  \int_0^t C U(t, s) (\lambda + A(s))x  e^{-\lambda s}\,ds  \| _{L_2}
\]
However, for $g \in L_2$ of norm one, using admissibility (!) of $C^*$ for $U(t, s)^*$, 
\begin{align*}
|\langle  \int_0^t \| C U(t, s) (\lambda + A(s))x  e^{-\lambda s}\| ds. g \rangle |
= & \; \left|\int_0^\tau \int_0^t  \langle C U(t, s) (\lambda + A(s))x  e^{-\lambda s}, g(t) \rangle \,ds\,dt\right|\\
= & \; \left|\int_0^\tau  \langle (\lambda + A(s))x \, e^{-\lambda s},  \int_s^\tau  U(t, s)^* C^* g(t) \rangle \,dt\,ds\right|\\
\le & \; M \int_0^\tau \| (\lambda + A(s))x \, e^{-\lambda s} \|_{H} \| g\|_{L_2(s, \tau)}\,ds\\
\le & \; M \int_0^\tau \| (\lambda + A(s))x \, e^{-\lambda s} \|_{H} \,ds
  \end{align*}
we obtain the {\bf Hautus condition},

\[
\delta \|x\| \le \tfrac{\|Cx\|}{\sqrt{2 \Re(\lambda)}} + M  \int_0^\tau \| (\lambda + A(s))x  e^{-\lambda s} \|_{H} \,ds \qquad Re(\lambda)>0, x \in \bigcap_s D(A(s))
\]

as a {\bf necessry condition} for exact observability. Remark: in case
$A(s)=A$ this collapses down to the Hautus test of Russell-Weiss.

\bigskip

\begin{remark}
If $A(s) \in B(H)$ for all $s$,  we do not know whether we have ``IFF'' as in the autonomous case.
\end{remark}

\subsubsection{The sufficient condition}

We consider the case when $C(s) = C$. Supposing that $C$ is admissible operator and satisfy the inequality :
\begin{align*}
\delta \|x\| 
\le & \; \tfrac{\|Cx\|}{\sqrt{\Re(\lambda)}} + \int_0^\tau \| (\lambda + A(s)) x  e^{-\lambda s} \|_{H} \,ds \\
\le & \; \tfrac{\|Cx\|}{\sqrt{\Re(\lambda)}} + \frac{1}{Re \lambda}\int_0^\tau \| (\lambda + A(s))Re \lambda x  e^{-\lambda s} \|_{H} \,ds 
\end{align*}

Here we assume $Re \lambda > \eta > 0$, and uniformly stable. If $Re \lambda < w$, we use 
\[
\|(\lambda+A(s))^{-1}x\| \leq \frac{C}{Re \lambda} \|x\|
\]
even it is still true for $C=0$. We have the following theorem

\begin{theorem} (Alan's)
Let $D: \Omega \rightarrow L(X,Y)$ be an operator-valued function analytic in an open set $\Omega \subset \mathbb{C}$. If $D(\lambda)$ is left (resp.right) invertible for every $\lambda \in \Omega$m then there is an analytic operator function $E: \Omega \rightarrow L(Y,X)$ such that 
\[
E(\lambda)D(\lambda) = I_X
\]
\end{theorem}
\begin{proof}
see (?)
\end{proof}
\begin{lemma}
If we have $\int_{a}^{b}\|f(s)x\|ds \geq \delta \|x\|$. There exist a subset $E \subset [a,b]$ such that $\lambda(E) > 0$ and $\|f(s)x\| \geq \frac{\delta}{b-a}\|x\|$
\end{lemma}

\begin{proof}
By using contradiction, it is easy to verify.
\end{proof}

Due to the lemma,  there exist a non-null set $E \subset [0, \tau]$ such that for all $\lambda \in \mathbb{C}$ and $s \in E$, we have

\[
\frac{\delta}{2}\|x\| \leq \frac{Cx}{\sqrt{Re \lambda}} + \frac{1}{Re \lambda} \|(A(s)+\lambda)e^{-\lambda s}x\|
\]

We have the map $x \mapsto (Cx, (A(s)+\lambda) e^{-\lambda s}x)$ is left-invertible for $0 \leq s \leq \infty$ and $\lambda \in \mathbb{C}$. Hence there exists the analytic functions $U_s(\lambda)$ and $V_s(\lambda)$ satisfying the equation 

\[
V_s(\lambda)(A(s)+\lambda)e^{-\lambda s}x + V_s(\lambda)Cx = I
\]  

\[
V_s(\lambda)e^{-\lambda s}x + V_s(\lambda)C(A(s)+\lambda)^{-1} = (A(s)+\lambda)^{-1}
\]

Intergating both sides we get

\[
0 + \frac{1}{2\pi i}\int_{|\lambda = r|}^{}V_s(\lambda)C(A(s)+\lambda)^{-1} d\lambda = I
\]

\[
\sum_{k}V_k(s)CA(s)^k = I
\]
Then the map $x \mapsto (CA(s)^kx)_{k \geq 0}$ is left-invertible.
Now we suppose that the system is not exactly observablem, then there exits a sequence ${z_n}_{ n \geq 1}$ such that $\|z_n\| = 1$ and $\langle z_n , Q z_n \rangle \rightarrow 0$ .

\begin{theorem}(Vitali's theorem)

Let $f_n(z)$ be a sequence of functions, each regular in a region $D$, let $ |f_n(z)| \leq M$ for every $n$ and $z$ in D, and let $f_n(z)$ tend to a limit as $ n \rightarrow \infty$ at a set of points having a limit point inside $D$. Then $f_n(z)$ tends uniformly to a limit in any region bounded by a contour interior to $D$, the limit therefore being an analytic function of $z$.

\end{theorem}

\[
f_n(\lambda) = C(t)U(\lambda, s )z_n
\]

$\|f_n(.)\|_{L^{\infty}} \leq M$ on an open set $D$. We have $f_n(t) \rightarrow 0$ on the set with accumulation points. By the Vitali's theoremm $f_n$ is uniformly convergent to $f$ on a compact subset of $D$.  

\begin{hypothese}
The evolution family $U(\lambda , s)$ is holomorphic. If $A(t)$ is bounded uniformly, could we infer that $U(\lambda , s)$ is holomorphic. 
\end{hypothese}

Then there exists a subsequence of functions $f_{n_{k_l}}$ such that $f_{n_{k_l}} \rightarrow 0$ uniformly on a compact subset of $D$. The contour integral of $f_n$ at the point $ \lambda = \omega$ is defined as

\[\frac{1}{2\pi i}\int_{D}^{}\frac{f_n(\lambda)}{(\lambda - \omega)} = f_n(\lambda)\]

Differentiating $f_n$ for $n$ times at the point $\lambda = \omega$ gives

\[ \frac{1}{2\pi i}\int_{D}^{}\frac{f_n(\lambda)}{(\lambda - \omega)^{n+1}} = (\frac{d}{d\lambda})^n f_n(\lambda) |_{\lambda = \omega}  = CA(t)^nU(\lambda,s)x_n |_{\lambda = \omega} = CA(\omega)^nU(\omega,s)x_n  \] 

Since $f_n(\lambda) \rightarrow 0$ uniformly, $CA(\omega)^nU(\omega,s)x_n \rightarrow 0$ uniformly, we have
\[ \sum_k W_kCA(s)U(\omega,s)x_n = x_n \]

\begin{theorem}
If $C$ is admissible and $A$ is boundedm then $C$ is bounded. 
\end{theorem}

\begin{proof}

First noting that, if $f$ is $C^1$ and $\alpha$-Holder function for $\alpha > 0$ then we have
\begin{equation}
f(0) = \frac{1}{\sigma}\int_{0}^{\sigma}f(s) ds - \int_{0}^{\sigma}\frac{1}{t^2}(\int_{0}^{t}f(t)-f(s)ds) dt  for all \sigma > 0
\end{equation} 

In fact, Let $\sigma > \epsilon > 0$. We have 
\begin{align*}
  \int_\epsilon^\sigma \frac 1{t^2} \left(\int_0^t f(t) - f(s)\, ds \right)\, dt 
   &= \int_\epsilon^\sigma \frac 1{t^2} \left(t f(t) - \int_0^t f(s)\, ds\right)\, dt\\
   &= \int_\epsilon^\sigma \frac{f(t)}t\, dt - \int_\epsilon^\sigma \int_0^t \frac{f(s)}{t^2}\, ds\,dt\\
   &= \int_\epsilon^\sigma \frac{f(t)}t\, dt -\int_0^\sigma \int_{\max\{s,\epsilon\}}^{\sigma} \frac{f(s)}{t^2}\, dt\, ds\\
   &= \int_\epsilon^\sigma \frac{f(t)}t\, dt - \int_0^\sigma f(s) \int_{\max\{s,\epsilon\}}^{\sigma} \frac{dt}{t^2}\, ds\\
  &= \int_\epsilon^\sigma \frac{f(t)}t \, dt - \int_0^\sigma f(s) \cdot \left(\frac 1{\max\{s,\epsilon\}} - \frac 1\sigma\right)\, ds\\
  &= \frac 1\sigma \int_0^\sigma f(s) \, ds - \frac 1\epsilon\int_0^\epsilon f(s)\, ds \\
\iff \frac 1\epsilon \int_0^\epsilon f(s)\, ds  
  &=  \frac 1 \sigma \int_0^\sigma f(s)\, ds - \int_\epsilon^\sigma \frac 1{t^2} \left(\int_0^t f(t) - f(s)\, ds \right)\, dt 
\end{align*}
Now let $\epsilon \to 0$, we get the result. Now if we take $f(t) = CU(\tau,t)x$ then

\[ Cx = \frac{1}{\sigma}\int_{0}^{\sigma}CU(\tau,t)xdt + \int_{0}^{\sigma}\frac{1}{t^2}(\int_{0}^{t}(\int_{s}^{t}CU(\tau,r)A(r)xdr)ds)dt \]

\[ = \frac{1}{\sigma}\int_{0}^{\sigma}CU(\tau,t)xdt + \int_{0}^{\sigma}\frac{1}{t^2}(\int_{0}^{t}(\int_{0}^{r}CU(\tau,r)A(r)xds)dr)dt \]

By triangle inequality,

\[ \|Cx\| \leq \frac{1}{\sqrt{\tau}}(\int_{0}^{\tau}\|CU(\tau,t)x\|^2)^{\frac{1}{2}} + \int_{0}^{\sigma}\frac{1}{t^2}\int_{0}^{t}r\|CU(\tau ,r)A(r)x\|dr dt \]
By Cauchy-Swart inequality and use the fact that $C$ be a admissible operator, we have
\begin{align*}
\int_{0}^{t}r\|CU(\tau ,r)A(r)x\|dr  
\leq & \;  (\int_{0}^{t}r^2 dr)^{\frac{1}{2}}(\int_{0}^{t}\|CU(\tau ,r)A(r)xdr\|)^{\frac{1}{2}} \\
\leq & \;  \frac{1}{\sqrt{\tau}}\int_{0}^{t}\|CU(\tau ,r)A(r)xdr\|)^{\frac{1}{2}} \leq \frac{M_{\tau}}{\sqrt{\tau}}\|x\|  
\end{align*}
So, we have

\[ \|Cx\| \leq (\frac{M_{\tau}}{\sqrt{\tau}}+ \frac{2M_{\tau}\sqrt{\tau}}{\sqrt{3}}\|A\|\|x\|) \]
Hence, C is a bounded operator.
\end{proof}

\section{Hautus test for the case of fix parameter}
\begin{lemma}
If $C$ is an admissible operator, i.e
\[\int_{0}^{\tau}\|CT(t)x\|^2 dt \leq M\|x\|^2\]
for all $x \in H$. Then we have
\[\int_{0}^{\tau}\|T(t)^*C^*y\|_H dt\| \leq M\sqrt{\tau}\|y\| \]
for all $y \in L^2(0,\tau;H)$
\end{lemma}
\begin{proof}
We have 
\begin{align*}
\bigl|\int_{0}^{\tau}\langle CT(t)x,y\rangle_U dt \bigr| = \bigl|\int_{0}^{T}\langle x,T(t)^*C^*y\rangle_U dt\bigr| 
= &  \bigl|\langle x,\int_{0}^{\tau}T(t)^*C^*y\rangle_{H} dt\bigr|  \\
\leq & (\int_{0}^{\tau}\|CT(t)x\|^2_U)^{\frac{1}{2}}\sqrt{\tau}|y\|_U \\
\leq & \sqrt{M} \|x\| \sqrt{\tau}\|y\|_U
  \end{align*}
Then
\[ \|\int_{0}^{\tau}T(t)^*C^*y dt\|_H = sup_x |\langle \frac{x}{\|x\|}, \int_{0}^{\tau}T(t)^*C^*y dt\rangle_H dt | \leq \sqrt{M}\sqrt{\tau} \|y\|_U  \]
\end{proof}

\begin{theorem}
Suppose the operators $A(t)$ is analytics in $L(H)$. Suppose that for all $ s \in [0, \tau]$ : $(C, e^{-tA(s)})_{t \geq 0}$ is exactly observable. Then we have $(C, U(t,.))$ is also exactly observable. 

\end{theorem}
Supposing that $(C, (A(s)+\lambda)^2e^{-tA(s)})_{t \geq 0}$ is exactly observable. We denote $D(t) =  (A(s)+\lambda)^2e^{-tA(s)} $. We obtain the inequality
\[
m^2\|x\|  \leq 2\|Cx. e^{-\lambda t}\|^2_{L^2} + \int_{0}^{\tau} \|(\lambda + (A(s)+\lambda)^2e^{-A(s) s} )e^{-(\lambda)s} x\| ds 
\]
By triangle inequality
\begin{align*}
m^2\|x\| 
\leq & \; 2\|Cx. e^{-\lambda t}\|^2_{L^2} + \int_{0}^{\tau} \|(\lambda e^{-\lambda s})x\|_{L^{\infty}} +  \int_{0}^{\tau} \|(A+\lambda)^2e^{-(A+\lambda)s} x\|_H ds   \\
\leq 2 & \; \|Cx. e^{-\lambda t}\|^2_{L^2} + (1-e^{-\tau \lambda}) \|x\| + \|(A+\lambda)x\|_H \int_{0}^{\tau} \|(A+\lambda)e^{-(A+\lambda)s} x\|_H ds   \\
 \leq 2 & \; \|Cx. e^{-\lambda t}\|^2_{L^2} + (1-e^{-\tau \lambda}) \| x\|  + \| (A+\lambda)x\| _H (1-e^{-\tau (A+\lambda)} \\
\leq 2 & \; \|Cx. e^{-\lambda t}\|^2_{L^2} + (1-e^{-\tau \lambda}) \| x\|  + \| (A+\lambda)x\| _H (1-e^{-\tau (A+\lambda)} 
\end{align*}
Therefore, we can refer that:
\[(m^2+e^{-\tau \lambda}-1)\|x\| \leq 2\|Cx. e^{-\lambda t}\|^2_{L^2} + \| (A+\lambda)x\| _H \]
The admissibility of observable operator $C$ means that for some $\tau > 0$, there exists $M \geq 0$ such that for any $x \in D(A(s))$,
\[\int_{0}^{\tau}\|CT(t)x\|^2dt \leq M.\|x\|^2 \]
Since $(C, e^{-tA(s)})_{t \geq 0}$ is exactly observable,

\begin{align*}
m^2\|x\| \leq \|CT(t)x\| 
= & \|C(e^{-\lambda t}x + \int_{0}^{t}T(t-s)(\lambda + A)e^{-\lambda s}ds)\|^2_{L^2} \\
= & 2\|Cx. e^{-\lambda t}\|^2_{L^2} + 2\| t\mapsto \int_{0}^{t}CT(t-s)(\lambda + A)xe^{-\lambda s} ds\|^2_{L^2)(H)}\\
= & 2\|Cx. e^{-\lambda t}\|^2_{L^2}  + sup_{\|h\|_{L^2}\leq 1}\int_{0}^{\tau}\int_{0}^{t}\langle(\lambda+A)xe^{-\lambda s}, T(t-s)^{*}C^*h(t) dt \rangle\\
= & 2\|Cx. e^{-\lambda t}\|^2_{L^2}  + sup_{\|h\|_{L^2}\leq 1}\int_{0}^{\tau}\langle(\lambda+A)xe^{-\lambda s}, \int_{s}^{\tau} T(t-s)^{*}C^*h(t) dt \rangle\\
\leq & 2\|Cx. e^{-\lambda t}\|^2_{L^2} + \int_{0}^{\tau} \|(\lambda + A(s)x)\|e^{-Re(\lambda)s}. \|\int_{s}^{\tau} T(t-s)^{*}C^*h(t)\|_H \\
\leq & 2\|Cx. e^{-\lambda t}\|^2_{L^2} + \int_{0}^{\tau} \|(\lambda + A(s)x)\|e^{-Re(\lambda)s}. M\|h\|_L^2
  \end{align*}
  
Finally, we obtain
\[\frac{m^2}{\frac{1}{2}+M^2}\|x\|^2 \leq \frac{2}{Re(\lambda)}\|Cx\|^2 + \frac{2}{(Re \lambda)^2}\|(\lambda +A)x\|^2\]

For all $\lambda$m there exists $\delta_{\lambda}$ positive such that
\[\delta_{\lambda}\|x\|^2 \leq \|Cx\|^2 + \|(\lambda+A)x\|^2\]
Moreover, the functions $Cx$ and $(\lambda+A)x)$ are holormophic over the whole complex plane. So that, 
the map $x \mapsto (Cx,(\lambda+A)x)$ is left-invertible and entire.

\begin{theorem} (Alan's)

Let $D: \Omega \rightarrow L(X,Y)$ be an operator-valued function analytic in an open set $\Omega \subset \mathbb{C}$. If $D(\lambda)$ is left (resp.right) invertible for every $\lambda \in \Omega$m then there is an analytic operator function $E: \Omega \rightarrow L(Y,X)$ such that 

\[
E(\lambda)D(\lambda) = I_X
\]

\end{theorem}

Hence there exists the analytic functions $U_s(\lambda)$ and $V_s(\lambda)$ satisfying the equation 

\[
U_s(\lambda)(A+\lambda)x + V_s(\lambda)Cx = x
\]  

\[
U_s(\lambda)x + V_s(\lambda)C(A+\lambda)^{-1}x = (A+\lambda)^{-1}x
\]

We represent $V_s(\lambda) = \sum_{k=0}^{+\infty}\lambda^j V_k(s)$. Intergating both sides we get

\[
0 + \frac{1}{2\pi i}\int_{|\lambda = r|}^{}V_s(\lambda)C(A+\lambda)^{-1} d\lambda = I
\]

\[
\sum_{k}V_k(s)CA^k = I
\]
Then the map $x \mapsto (CA^kx)_{k \geq 0}$ is left-invertible.

Now we suppose that the system is not exactly observable, i.e there does not exist $m > 0$ such that
\[ \int_{0}^{\tau}\|U(\tau,t)^*Cz\|^2dt \geq m\|z\|^2 \] 
for all $z \in H$, then there exists a sequence ${z_n}_{n \geq 1}$ such that $\|z_n\| = 1$ and $\langle z_n, Qz_n \rangle \rightarrow 0$ where $Q = \int_{0}^{+\infty}C(s)U(\tau,s)U^*(\tau,s)C^*(s)ds$

\begin{theorem}(Vitali's theorem)

Let $f_n(z)$ be a sequence of functions, each regular in a region $D$, let $ |f_n(z)| \leq M$ for every $n$ and $z$ in D, and let $f_n(z)$ tend to a limit as $ n \rightarrow \infty$ at a set of points having a limit point inside $D$. Then $f_n(z)$ tends uniformly to a limit in any region bounded by a contour interior to $D$, the limit therefore being an analytic function of $z$.

\end{theorem}

\[
f_n(\lambda) = C(t)U(\lambda, s )z_n
\]

$\|f_n(.)\|_{L^{\infty}} \leq M$ on an open set $D$. We have $f_n(t) \rightarrow 0$ on the set with accumulation points. By the Vitali's theoremm $f_n$ is uniformly convergent to $f$ on a compact subset of $D$.

\begin{hypothese}
The evolution family $U(\lambda , s)$ is holomorphic. If $A(t)$ is bounded uniformly, could we infer that $U(\lambda , s)$ is holomorphic. 
\end{hypothese}

Then there exists a subsequence of functions $f_{n_{k_1}}$ such that $f_{n_{k_l}} \rightarrow 0$ uniformly on a compact subset of $D$. The contour integral of $f_n$ at the point $ \lambda = \omega$ is defined as

\[\frac{1}{2\pi i}\int_{D}^{}\frac{f_n(\lambda)}{(\lambda - \omega)} = f_n(\lambda)\]

Differentiating $f_n$ for $n$ times at the point $\lambda = \omega$ gives

\[ \frac{1}{2\pi i}\int_{D}^{}\frac{f_n(\lambda)}{(\lambda - \omega)^{k+1}} = (\frac{d}{d\lambda})^n f_n(\lambda) |_{\lambda = \omega}  = CA(t)^nU(\lambda,s)x_n |_{\lambda = \omega} = CA(\omega)^nU(\omega,s)x_n  \] 

$f_n(\lambda) \rightarrow 0$ uniformly on $\delta D$. 
So that
\[ \|CA^kU(\omega,s)x_n\| = \|\frac{1}{2\pi i}\int_{\delta D}^{}\frac{f_n(\lambda)}{(\lambda - \omega)^{k+1}} \| \leq \|\int_{\delta D}^{}\frac{max_{D}|f_n(\lambda)| }{(\lambda - \omega)^{k+1}}d\lambda \| \leq \alpha_n . \frac{2\pi r}{r^{k+1}} = \alpha_n \frac{2\pi}{r^n}  \]
where $\alpha_n \rightarrow 0$ when $n \rightarrow +\infty$, and $r > 1$. Therefore, $CA^kU(\omega,s)x_n \rightarrow 0$ when $n \rightarrow +\infty$. Using the estimation $\sum_{k}V_k(s)CA^k = I$
\[ \sum_k V_kCA^kU(\omega,s)x_n = U(\omega,s)x_n \]
We finally obtain  $ \|U(\omega,s)x_n\| \rightarrow 0$ when $n \rightarrow +\infty$. That is a contradiction because we already assumed that $\|x_n\| =1$ for all $n$. As a result, the system $(C,A)$ is exact observable.


\begin{thebibliography}{00}

\bibitem{Pazy}
A. Pazy, \textit{Semigroups of Linear Operators and Applications to Partial Differential Equations},
\textit{ESAIM}, 1983.

\bibitem{Jacob1}
B. Jacob,  R. Schnaubelt, \textit{Observability of polynomially stable systems}, 
\textit{Control Lett.}, 56 (2007), pp. 277-284.

\bibitem{Jacob2}
B. Jacob, H. Zwart, \textit{ observability of diagonal systems with a finite-dimensional output operator}, 
\textit{Control Lett.}, 43 (2001), pp. 101-109.

\bibitem{Jacob3}
B. Jacob, H. Zwart, \textit{On the Hautus test for exponentially stable $C_0$-groups}, 
\textit{SIAM J.Control Optim}, vol. 48, No.3, pp 1275-1288.

\bibitem{russell}
D. L. Russell, G. Weiss, \textit{A general necessary condition for exact observability}, 
\textit{SIAM J. Control Optim}, 32 (1), 1–23, 1994.

\bibitem{tucsnack}
M. Tucsnak, G. Weiss, \textit{Observation and Control for Operator Semigroups}, 
Birkhäuser Verlag, Basel, 2009.

\end{thebibliography}
\end{document}